\documentclass[11pt, letterpaper]{amsart}
\usepackage[english]{babel}
\usepackage{amssymb,amsmath,amsthm,mathrsfs,bm}
\usepackage{color}
\usepackage[margin=1in]{geometry}
\usepackage{stmaryrd}

\newtheorem{theorem}[subsection]{Theorem}

\newtheorem{definition}[subsection]{Definition}
\newtheorem{lemma}[subsection]{Lemma}
\newtheorem{remark}[subsection]{Remark}
\newtheorem{proposition}[subsection]{Proposition}
\newtheorem{corollary}[subsection]{Corollary}

\newtheorem*{claim*}{Claim}
\newtheorem*{theorem*}{Theorem}
\def\bal{\begin{aligned}}
\def\eal{\end{aligned}}
\def\be{\begin{equation}}
\def\ee{\end{equation}}
\def\bcs{\begin{cases}}
\def\ecs{\end{cases}}
\def\={\;=\;}
\def\+{\,+\,}
\def\-{\,-\,}

\def\Z{{\mathbb Z}}

\def\Q{{\mathbb Q}}
\def\R{{\mathbb R}}

\def\sA{\mathcal{A}} 

\def\sO{\mathcal{O}} 

\def\sG{\mathcal{G}}

\def\supp{{\rm supp}}

\def\lb{\llbracket}
\def\rb{\rrbracket}
\def\pow#1{\llbracket #1\rrbracket}
\def\ord{\mathrm{ord}}

\def\cartier{\mathscr{C}_p}
\def\fil{\mathscr{F}}

\def\v#1{{\bf #1}}

\def\is{\equiv}
\def\mod#1{({\rm mod}\ #1)}
\def\hat{\widehat}

\definecolor{lightgrey}{rgb}{0.8, 0.84, 0.8}

\title{Frobenius structure and $p$-adic zeta values
}
\author{Frits Beukers, Masha Vlasenko}

\address{Utrecht University}
\email{f.beukers@uu.nl}

\address{Kyiv School of Economics}
\email{maria.vlasenko@kse.org.ua}

\thanks{
Work of Masha Vlasenko was supported by the National Science Centre of 
Poland (NCN), grant UMO-2020/39/B/ST1/00940.}

\date{\today}
\begin{document}

\begin{abstract}
For differential operators of Calabi-Yau type, Candelas, De la Ossa and van Straten conjecture
the appearance of $p$-adic zeta values in the matrix entries of their $p$-adic Frobenius structure expressed in the standard basis of solutions near a point of maximal unipotent local monodromy. We prove that this phenomenon holds for simplicial and hyperoctahedral families of Calabi-Yau hypersurfaces in $n$ dimensions, in which case the limits of the Frobenius matrix entries are rational linear combinations of products of $\zeta_p(k)$ with $1< k < n$.
\medskip

\end{abstract}

\maketitle

\section{Introduction and main results}
Let $L \in \Q(t)[\frac{d}{dt}]$ be a linear differential operator of order $n$ 
and let $p$ be a prime number.  A $p$-adic Frobenius structure of $L$ is given by a
differential operator with coefficients in $p$-adic analytic elements $\sA\in E_p[\frac{d}{dt}]$ such that for any solution $y(t)$ of $L(y) = 0$ the function $\sA(y(t^p))$ is again a solution of $L$. Here the field of $p$-adic analytic elements $E_p$ is the completion  of $\Q(t)$ with respect to the $p$-adic Gauss norm $\left| \frac{\sum_i a_i t^i}{\sum_j b_j t^j}\right| = \frac{\max_{i}(|a_i|_p)}{\max_j(|b_j|_p)}$.
The concept of Frobenius structure was introduced by Bernard Dwork. The definition we use in this paper is a variation on it, see \cite[Remark 1.3]{IN}. Existence of a $p$-adic Frobenius structure is a strong property, we only expect it for differential operators $L$ which arise from the Gauss--Manin connection in algebraic geometry. 
Section~\cite[\S17]{Kedlaya} of Kedlaya's book is a useful modern reference on Frobenius structures on differential modules.  

We will consider differential operators whose local monodromy around $t=0$ is maximally unipotent. Denote $\theta = t \frac{d}{dt}$ and assume that 
\[
L = \theta^n + a_1 \theta^{n-1} + \ldots + a_{n-1} \theta + a_n 
\]
with $a_j(t)\in\Q(t)$ and $a_j(0)=0$ for $j=1,\ldots,n$. 
We call such operators \emph{MUM-type operators}.
They have a unique basis of solutions of the form 
\[
y_i = F_0 \frac{\log^i t}{i!} + F_1 \frac{\log^{i-1} t}{(i-1)!} + \ldots + F_{i-1} \log t + F_i, \quad i=0,\ldots, n-1, 
\]
where $F_i \in \Q\lb t \rb$, $F_0(0)=1$ and $F_i(0)=0$ for $i>0$. We call it  \emph{the standard basis}.

\begin{definition} Let $R\subset\Z_p\lb t\rb$ be a $p$-adically complete subring. We say that $L$ has a $p$-adic Frobenius structure over $R$ if there exists a
differential operator 
$\sA=\sum_{j=0}^{n-1} A_j(t) \theta^j \in R[\theta]$ with $A_0(0)=1$ such that
for every $i=0,1,\ldots,n-1$ the composition $\sA(y_i(t^p))$ is again a solution of this differential equation.
\end{definition}

\begin{remark}\label{uniqueness-remark}
In principle it would be possible to consider more general Frobenius lifts $t^\sigma$ such that
$t^\sigma\is t^p\mod{p}$. For example, this is done in \cite{IN}, where we have the ring $R=\Z_p\lb t\rb$
and the so-called excellent Frobenius lift. In the present paper we restrict ourselves to $t^\sigma=t^p$ in order to simplify the exposition.

If $L$ is irreducible in $\Q(t)[\theta]$ and has a Frobenius structure over a ring  $R\subset E_p$ consisting of $p$-adic analytic elements, then the operator $\sA$ is uniquely determined. See \cite{Dwork89}. 

When $R=\Z_p\lb t\rb$ this uniqueness is not true anymore. We will give an example in Section~\ref{sec:Frobenius-structure}.
\end{remark}

The following proposition gives a more explicit description of the Frobenius action.

\begin{proposition}\label{special-Frobenius-action}
Suppose a MUM-type operator $L$ has a Frobenius structure with corresponding operator
$\sA$. Then there exist $\alpha_0,\alpha_1, \ldots,\alpha_{n-1}\in\Z_p$ such that
\be\label{Frob-str-equation}
\sA(y_i(t^p)) = p^i \sum_{j=0}^i \alpha_j \, y_{i-j}(t), \quad i=0,\ldots, n-1.
\ee 
Moreover, $\alpha_j=A_j(0)$ for $j=0,\ldots,n-1$. In particular $\alpha_0=A_0(0)=1$.
\end{proposition}
This is the content of Remark 1.2 in \cite{IN} and it is proven in \cite[\S2]{IN}. 

We note, given the $\alpha_i$, equation~\eqref{Frob-str-equation} allows us to express the coefficients $A_0(t), \ldots, A_{n-1}(t)$ in terms of the power series $F_0(t), \ldots, F_{n-1}(t)$. Therefore, for any set of $p$-adic constants $\alpha_0,\ldots,\alpha_{n-1}$ one has a unique collection of power series $A_0,\ldots,A_{n-1} \in \Q_p\lb t \rb$ for which  equation~\eqref{Frob-str-equation} holds. The point of the above definition is the existence of a set of special values of $\alpha_i$ for which all $A_j$ belong to the specific ring $R\subset\Z_p\lb t\rb$.

In this paper we are going to compute the constants 
$\alpha_1, \ldots, \alpha_{n-1}$ defining the Frobenius structure for particular differential operators $L$. Our motivation came from the experimental results for the 4th order differential operators of Calabi-Yau type, which were reported in~\cite{DucoCreswick} and stated as a conjecture in~\cite[\S4.4]{COS21}. Namely, Candelas, de la Ossa and van Straten conjecture that for this class of operators one has $\alpha_1=\alpha_2=0$ and $\alpha_3$ is a rational multiple of $\zeta_p(3)$, the value of the $p$-adic zeta function at $3$. The only proved case is Shapiro's computation in~\cite{Sh09,Sh12} for the operator
\[
L = \theta^4 - (5 t)^5 (\theta+1)(\theta+2)(\theta+3)(\theta+4).
\]
Their conjecture is a natural $p$-adic analog of the monodromy conjecture stating that the matrix entries of the monodromy representation of Calabi--Yau differential operators of order $4$ in the standard basis of solutions belong to the ring generated over $\Q$ by $\zeta(3)/(2 \pi i)^3$, see~\cite[\S 2.7]{DucoCYOperators}. More generally, for Calabi--Yau operators of order $n$ the matrix entries of the monodromy representation in the standard basis of solutions are conjectured to take values in the ring generated over $\overline \Q$ by the numbers $\zeta(k)/(2 \pi i)^k$ with $2 \le k < n$, see~\cite[\S 2.7, Conjecture 2]{DucoCYOperators}. By analogy, we may expect the appearance of the respective $\zeta_p(k)$ among the constants $\alpha_j$ defining their Frobenius structure. Our main results, Theorems~\ref{main1} and~\ref{main2} below, support this expectation.

\begin{theorem}\label{main1} For any $p > n+1$ the differential operator
\[
L =  \theta^n - ((n+1)t)^{n+1} (\theta+1)\ldots(\theta+n)
\]  
has a $p$-adic Frobenius structure defined over the ring
\[
R = \text{ $p$-adic completion of }\; \Z[t,1/D(t)],
\]
where $D(t)=1-((n+1)t)^{n+1}$ is the leading coefficient of $L$.
The constants $\alpha_{j}$ defining the Frobenius structure via~\eqref{Frob-str-equation} are given by 
\[
\alpha_{j}= \text{ coefficient of } x^{j} \text{ in } \frac{\Gamma_p(x)}{\Gamma_p(x/(n+1))^{n+1}}, \quad 1 \le j \le n-1,
\]
where $\Gamma_p$ is Morita's $p$-adic $\Gamma$-function.
\end{theorem}

Similarly to the classical gamma function, the expansion coefficients of $\Gamma_p(x)$ at $x=0$ involve $p$-adic
zeta values. Namely, by~\cite[\S58 and \S61]{Sch} we have
\be\label{gamma-zeta-relation}
\log \Gamma_p(x)=\Gamma'_p(0)x-\sum_{m\ge2}\frac{\zeta_p(m)}{m}x^m.
\ee
The numbers $\zeta_p(m)$ are called $p$-adic zeta values and they are given by 
\[
\zeta_p(m) = \text{ $p$-adic limit of }-(1-p^{n-1})\left.\frac{B_n}{n}\right|_{n=1-m+(p-1)p^r} \text{ as } r\to\infty,
\]
where $B_n$ is the $n$-th Bernoulli number.
In particular we find that $\zeta_p(m)=0$ when $m$ is even. For a more extended discussion of the values $\zeta_p(m)$, in particular the expansion~\eqref{gamma-zeta-relation}, we recommend \cite[Appendix B]{COS21}. 

As an illustration, we compute some expansions in Theorem~\ref{main1} for small $n$:
\[\bal
\frac{\Gamma_p(x)}{\Gamma_p(x/5)^5} &= 1 - \frac{8}{25} \zeta_p(3)  x^3 + O(x^4) \\
\frac{\Gamma_p(x)}{\Gamma_p(x/6)^6} &= 1 - \frac{35}{108} \zeta_p(3) x^3 + O(x^5) \\
\frac{\Gamma_p(x)}{\Gamma_p(x/7)^7} = 1 - &\frac{16}{49} \zeta_p(3) x^3 - \frac{480}{2401} \zeta_p(5) x^5 + O(x^6) \\
\frac{\Gamma_p(x)}{\Gamma_p(x/8)^8} = 1 - \frac{21}{64} \zeta_p(3) & x^3 - \frac{819}{4096} \zeta_p(5) x^5 + 
\frac{441}{8192} \zeta_p(3)^2 x^6 + O(x^7) \\
\eal\]
The differential operator in this theorem is a Picard-Fuchs operator of a family of algebraic varieties parametrized by $t$. Particularly, such a family is given by the equation $g(\v x)=1/t$ with the Laurent polynomial 
\be\label{simplicial-g}
g(\v x)=x_1+\cdots+x_n+\frac{1}{x_1\cdots x_n}.
\ee
We call it the \emph{simplicial family} referring to the shape of the Newton polytope of $g(\v x)$. One can easily see that the operator $L$ in Theorem~\ref{main1} annihilates the period integral
\[
y_0(t) = \frac1{(2 \pi i)^n} \oint \ldots \oint \frac1{1-t g(\v x)}\frac{dx_1}{x_1} \ldots \frac{dx_n}{x_n} = \sum_{k=0}^\infty \frac{(k(n+1))!}{k!^{n+1}} t^{(n+1)k}.
\] 

In the case $n=4$ of the famous quintic threefold we see that we get $\alpha_3=-8\zeta_p(3)/25$. In 
\cite[Conjecture 4.4]{COS21} the authors conjecture the limit value $-40\zeta_p(3)$.
The difference by the factor $5^3$ can
be explained by the fact that our differential equation comes from the differential equation in
\cite[\S 6.1]{COS21} after replacing $\phi$ by $t^5$.

Our second example shows that $p$-adic zeta values may also occur among the Frobenius structure constants of non-hypergeometric Picard-Fuchs differential operators. Consider the Laurent polynomial
\be\label{hyperoctahedral-g}
g(\v x)=x_1+\frac{1}{x_1}+x_2+\frac{1}{x_2}+\cdots+x_n+\frac{1}{x_n}.
\ee
The corresponding family of varieties $g(\v x)=1/t$ is called the \emph{hyperoctahedral family}, also
after the shape of its Newton polytope. In~\cite[\S 6]{IN} we constructed a differential operator $L \in \Z[t][\theta]$ of order $n$ which annihilates the respective  period function 
\[
y_0(t) = \frac1{(2 \pi i)^n} \oint \ldots \oint \frac1{1-t g(\v x)}\frac{dx_1}{x_1} \ldots \frac{dx_n}{x_n} = \sum_{k=0}^\infty t^{2k} \sum_{k_1+\ldots+k_n=k} \frac{(2k)!}{(k_1!\cdots k_n!)^2}.
\]
This operator $L$ is not of hypergeometric origin,
but it is of MUM-type. Here are a few examples. For $n=4$ we have
\begin{eqnarray*}
L&=&(1 - 80 t^2 + 1024 t^4) \theta^4 +(-320 t^2 + 8192 t^4)\theta^3\\
&&+(-528 t^2 + 23552 t^4)\theta^2 +(-416 t^2 + 28672 t^4)\theta -128t^2+12288t^4,
\end{eqnarray*}
which is equation $\#16$ from \cite{AESZ10} with $z$ replaced by $t^2$.
When $n=5$ we have
\begin{eqnarray*}
L&=&(1 - 140 t^2 + 4144 t^4 - 14400 t^6) \theta^5+(-700 t^2 + 41440 t^4 - 216000 t^6)\theta^4\\
&&(-1568 t^2 + 163280 t^4 - 1224000 t^6) \theta^3+(-1904 t^2 + 316640 t^4 - 3240000 t^6)\theta^2\\
&&(-1216 t^2 + 300096 t^4 - 3945600 t^6)\theta-320 t^2 + 109440 t^4 - 1728000 t^6.
\end{eqnarray*}
We are indebted to Jaques-Arthur Weil and Alexandre Goyer who verified that $L$
is irreducible in $\Q(t)[\theta]$ for $2\le n\le 30$. We conjecture that it is irreducible for all $n\ge2$.
 
\begin{theorem}\label{main2}
Let $L \in \Z[t][\theta]$ be a differential operator of order $n$, which was constructed
in \cite[Proposition 6.1]{IN} corresponding to the hyperoctahedral family in $n$ dimensions; this operator is of MUM type and can be chosen so that its leading coefficient $D(t) \in \Z[t]$ has $D(0)=1$. If $L$ is irreducible in $\Q(t)[\theta]$, then for any $p>n$ operator $L$ has a $p$-adic Frobenius structure. This Frobenius structure is defined over the ring
\[
R = \text{ $p$-adic completion of }\; \Z[t,1/D(t)].
\]
For  $1 \le j \le n-1$, the $p$-adic constants $\alpha_j$ defining the Frobenius structure via~\eqref{Frob-str-equation} are given as the coefficients of $x^j$ in the expansion of
\[
\Gamma_p(x)\, e^{- \Gamma_p'(0)x} = \exp\left( - \sum_{m \ge 2} \frac{\zeta_p(m)}m x^m \right).
\]
\end{theorem}

\medskip

For small $j$ this theorem gives
\[\bal
&\alpha_1=\alpha_2=0, \quad \alpha_3=-\zeta_p(3)/3, \quad \alpha_4=0,\\
&\alpha_5= -\zeta_p(5)/5, \quad \alpha_6= \zeta_p(3)^2/18,\quad \alpha_7=-\zeta_p(7)/7,\\
&\quad \alpha_8=\zeta_p(3)\zeta_p(5)/15,\quad \alpha_9=-(18 \zeta_p(9)+\zeta_p(3)^3)/162.
\eal\]

Interestingly enough, in this case any particular $\alpha_j$ does not depend on $n$. In
the case $n=4$ we see that $\alpha_3=-\zeta_p(3)/3$, while \cite[Conjecture 4.4]{COS21} predicts
the limit value $-8\zeta_p(3)/3$. The extra factor $8$ comes from the fact that our differential
equation is $\#16$ from \cite{AESZ10} with $z$ replaced by $t^2$.

Although Theorems \ref{main1} and \ref{main2} deal with particular examples,
our paper presents a method
which allows to determine the $p$-adic Frobenius structure of a differential equation of
Picard--Fuchs type explicitly. There are very few results of this kind in the literature.
Our major predecessor is paper~\cite{Sh09} which was already mentioned here.
Secondly, in~\cite{Ked19} Kedlaya gives an explicit description of the Frobenius structure
for hypergeometric differential operators whose local exponents at $t=0$ are distinct.
He uses Dwork's construction which is particular for hypergeometric differential operators.
Example 4.1.1 in loc. cit. deals with the case $n=4$ of our Theorem~\ref{main1}. 
There Kedlaya suggests to treat the cases of equal local exponents by $p$-adic interpolation.
Our method is essentially different. It is based on a generalization in~\cite{DCIII} of Katz's explicit description of unit root crystals. The key result is a decomposition of a module of differential $n$-forms on the complement of the hypersurface $g(\v x)=1/t$ into a sum of a finite rank free module and a module of forms whose expansion coefficients satisfy certain congruence conditions. We explain this in more detail in Section~\ref{sec:Frobenius-structure}. With some modifications in the computational steps,
we expect that many other differential operators $L$ can be treated along the same lines. 

{\bf Acknowledgement}: We would like to thank the referee for a careful reading of the manuscript and for several valuable observations which have helped improve the paper. We also thank Don Zagier for pointing out a simpler form of expression for $\alpha_j$ in Theorem~\ref{main2}.

\section{Construction of the Frobenius structure}\label{sec:Frobenius-structure}

We consider families of varieties parametrized by $t$. More particularly, they are given by zero sets of  \[
f(\v x)=1- t g(\v x),\]
where $g(\v x)$ is a Laurent polynomial in the variables $x_1,\ldots,x_n$ and coefficients in $\Z$. The support of $f(\v x)$ is denoted by $\supp(f) \subset \Z^n$, this is the finite set of exponent vectors of the monomials occuring in $f$. The convex hull of $\supp(f)$ is called the Newton polytope of $f$, it is denoted by $\Delta \subset \R^n$. We assume that $\Delta$ is
\emph{reflexive}. This condition means that $\Delta$ is of maximal dimension, contains $\v 0$
in its topological interior $\Delta^\circ$, and each codimension one face 
$\tau \subset \Delta$ can be given by an equation $\ell_\tau(\v u)=1$ where 
$\ell_\tau(\v u)=\sum_{i=1}^n \ell_i u_i$ is a linear functional with integral coefficients 
$\ell_i \in \Z$. For a vector $\v u \in \Z^n$ we denote by $\deg(\v u)$ the minimal number $d \ge 0$ such that $\v u \in d \Delta$. The condition that $\Delta$ is reflexive implies that $d\in\Z$. 

We start with the ring 
\[
R=\Z[t,1/D_f(t)],
\]
where $D_f(t)\in\Z[t]$ is some polynomial with $D_f(0)\ne0$. The choice of the polynomial $D_f(t)$ will depend on $f(\v x)$, which explains our notation for it. The elements of $R$ can be expanded into formal power series in $\Q\lb t \rb$ and $\ord_t$ denotes their $t$-adic valuation, the smallest degree of the terms in the power series. A Laurent polynomial $h(\v x) = \sum_\v u h_\v u \v x^\v u \in R[x_1^{\pm 1},\ldots,x_n^{\pm 1}]$ is called \emph{admissible} if $\ord_t(h_\v u) \ge \deg(\v u)$ for every $\v u \in \supp(h)$. For example, $f(\v x) = 1 - t g(\v x)$ is an admissible polynomial. We define the $R$-module of admissible rational functions as
\[
\sO_f = \left\{ (m-1)!\frac{h(\v x)}{f(\v x)^m} \;\Big|\; m \ge 1, h(\v x) 
\text{ is admissible, }\supp(h) \subset m \Delta \right\}.
\]
The submodule $\sO_f^\circ \subset \sO_f$ is defined by the stronger condition $\supp(h) 
\subset m \Delta^\circ$. The submodule of derivatives $d\sO_f \subset \sO_f$ is the $R$-module
generated by $x_i \frac{\partial}{\partial x_i}(\eta)$ for $\eta \in \sO_f$ and $i = 1, \ldots, n$. 

Note that $\theta=t\frac{d}{dt}$ is a derivation on $R$. It extends to $\sO_f$ by acting on the coefficients of rational functions by the usual rules of differential calculus. The action of $\theta$ preserves the submodules $\sO_f^\circ$ and $d \sO_f$. In particular, $\sO_f^\circ/d\sO_f$ is a differential $R$-module. Although $d \sO_f \not\subset \sO_f^\circ$, we shall write $\sO_f^\circ / d\sO_f$ in place 
of $\sO_f^\circ / (\sO_f^\circ \cap d\sO_f)$ here and in other similar situations. Though we will not use it here, we like to note that multiplying elements of $\sO_f$ by $\frac{d x_1}{x_1}\wedge \cdots \wedge \frac{d x_n}{x_n}$ one obtains differential forms on the complement of the hypersurface $f(\v x)=0$ in the $n$-dimensional torus. Elements of $d\sO_f$ are then mapped to exact forms and the above mentioned differential module represents the Gauss-Manin connection.   

We now exploit the possible symmetries in $g(\v x)$. Let $\sG$ be a
finite group of monomial substitutions of $x_1,\ldots,x_n$ such that $g(\v x^\gamma)=g(\v x)$ for every
$\gamma\in\sG$.
By a monomial substitution we mean a substitution $x_i\to \v x^{\v w_i}$ for $i=1,\ldots,n$, where
$\v w_i\in\Z^n$ and $\det(\v w_1,\ldots,\v w_n)=\pm1$. Instead of $\sO_f^\circ$ we
can now take the symmetric part $(\sO_f^\circ)^{\sG}$ consisting of $\sG$-invariant elements and consider the quotient module $(\sO_f^\circ)^{\sG}/d\sO_f$. Since derivations of $R$ commute with monomial substitutions, this is again a differential $R$-module. 

Let us recall a definition from~\cite[\S 4]{IN}. 

\begin{definition}\label{cyclicMUM}
We shall say that the $R$-module $M=(\sO_f^\circ)^\sG /d \sO_f$ is a cyclic $\theta$-module
of rank $n$ if 
\begin{itemize}
\item[(i)]
for every $m\ge0$ and for every $\sG$-invariant admissible
Laurent polynomial $h(\v x)$ supported in $m\Delta$ we have
\[
m!\frac{h(\v x)}{f(\v x)^{m+1}}\is\sum_{j=0}^{\min(m,n-1)}b_j(t)\theta^j(1/f)
\quad \mod{d \sO_{f}}
\]
with $b_j(t)\in R$ for all $j$;
\item[(ii)]
the monic differential operator $L \in R[\theta]$ of order $n$ such that $L(1/f)\in d\sO_f$ is irreducible in 
$\Q(t)[\theta]$.
\end{itemize}

Note that existence of $L$ in~(ii) follows from~(i). We call $L$ the Picard-Fuchs differential operator associated to $M$. Write $L=\theta^n + \sum_{j=1}^{n}a_j(t)\theta^{n-j}$ with $a_j\in R$
for all $j$. 
We shall say that $M$ is of MUM-type (maximally unipotent local monodromy) if $a_j(0)=0$
for all $1 \le j\le n$.
\end{definition}

Our ring $R$ is a subring of $\Q(t)$. It is not hard to show that condition~(ii) above implies that no $R$-linear combination of $\theta^j(1/f)$, $0 \le j < n$ belongs to $d\sO_f$, see~\cite[Lemma 4.5]{IN}. Hence $M$ is a free $R$-module where these $n$ elements can be taken as a basis. 

Consider the simplicial family with $g(\v x)$ given in~\eqref{simplicial-g} and take $D_f(t)=(n+1)(1-((n+1)t)^{n+1})$. In~\cite[Proposition 5.1]{IN} we showed that the module $\sO_f^\circ/d \sO_f$ is a cyclic theta module of rank~$n$ whose Picard-Fuchs operator is the operator $L$ from Theorem~\ref{main1}. Though there are many symmetries of $g(\v x)$, in this example we get a cyclic $\theta$-module without taking the symmetric part. Here $\sG=\{1\}$. 

Consider the hyperoctahedral family with $g(\v x)$ given in~\eqref{hyperoctahedral-g}. Here we take the group $\sG$ generated by the permutations of $x_1,\ldots,x_n$ and $x_i\to x_i^{\pm1}$ for $i=1,\ldots,n$. In~\cite[Proposition 6.1]{IN} we showed existence of $D_f(t) \in \Z[t]$ for which condition~(i) of Definition~\ref{cyclicMUM} is satisfied and a MUM-type operator $L\in R[\theta]$ of order $n$ such that $L(1/f)\in d\sO_f$. We couldn't prove irreducibilty of $L$ but we checked it for small $n$ with computer algebra tools. This is essentially the differential operator from Theorem~\ref{main2}. 

Let $p$ be a prime such that $p \nmid D_f(0)$ and let 
\[
R = \text{ $p$-adic completion of }\Z[t,1/D_f(t)].
\]
Let $\hat\sO_f$ denote the $p$-adic completion of the module of admissible rational functions $\sO_f$.
In~\cite{DCI} and~\cite{DCIII} we introduced an $R$-linear operator
\[
\cartier: \hat\sO_f \to \hat\sO_{f^\sigma}
\] 
which sends $d \hat \sO_f$ to $d \hat \sO_{f^\sigma}$ and commutes with $\theta$. Here $f^\sigma(\v x)$ is the polynomial $f(\v x)$ with the Frobenius lift $\sigma:t\to t^p$ applied to its coefficients, which in our current situation is $f^\sigma(\v x)=1-t^p g(\v x)$. We call $\cartier$ the Cartier operator. It is related to the Frobenius map on exponential modules constructed by Dwork, see \cite[Theorem A.5]{DCI}, with an advantage that our $\cartier$ is defined directly on rational functions (or differential forms).

In \cite[Proposition 4.2]{IN} we show that if $(\sO^\circ_f)^\sG/d \sO_f$ is a cyclic $\theta$-module of rank $n$ and if $p \ge n$ does not divide $\#\sG\times D_f(0)$, then there exist 
$\lambda_i(t)\in p^i\Z_p\lb t \rb$ such that
\[
\cartier(1/f)\is \sum_{i=0}^{n-1}\lambda_i(t)(\theta^i(1/f))^\sigma
\quad \mod{d\hat\sO_{f^\sigma}}.
\]
The proof of this proposition actually shows that $\lambda_i(t)\in p^iR$, where $R$ is our 
$p$-adic completion of $\Z[t,1/D_f(t)]$. In addition, since the Frobenius lift is given by $t\to t^p$, we have that $(\theta^i(1/f))^\sigma=p^{-i}\theta^i(1/f^\sigma)$. Thus, by setting $A_i(t)=p^{-i}\lambda_i(t)$,
we arrive at the following. 
\begin{proposition}\label{cartier-mod-exact}
Let $(\sO_f^\circ)^\sG / d \sO_f$ be a cyclic $\theta$-module of rank $n$ over the ring $\Z[t,1/D_f(t)]$. Suppose $p \ge n$, $p\nmid\#\sG\times D_f(0)$ and $R$ is the $p$-adic completion of $\Z[t,1/D_f(t)]$. 
Then there exist $A_j(t)\in R$ such that 
\be\label{cartier-defines-A-i}
 \cartier (1/f) = \sum_{j=0}^{n-1} A_j(t) \, \theta^j(1/f^\sigma) \; \mod {d \hat \sO_{f^\sigma}}.
\ee
\end{proposition}

As we mentioned earlier, the derivatives $\theta^j(1/f)$, $0 \le j \le n-1$ are linearly independent modulo $d \sO_{f}$. The same holds for $f^\sigma$ instead of $f$, see~\cite[Remark 4.7]{IN}. However at this point it is not clear that these elements are linearly independent modulo the $p$-adic completion $d \hat \sO_{f^\sigma}$. In particular, it is not clear that the identity~\eqref{cartier-defines-A-i} determines the coefficients $A_i(t)$ uniquely. Our way out of this problem is to assume that the so-called $n$-th \emph{Hasse--Witt
condition} holds, see~\cite[Lemma 4.8]{IN}. The Hasse--Witt conditions are defined in~\cite{DCIII}, and they can be verified for the simplicial and hyperoctahedral families by \cite[Theorem 3.3]{IN}. Namely, Theorem ~3.3 in \emph{loc. cit.} states that if all proper faces of the Newton polytope $\Delta$ are \emph{simplices of volume}~1 and all coefficients of $g(\v x)$ at the vertices of $\Delta$ are in $\Z_p^\times$, then the $k$-th Hasse-Witt condition holds over the larger ring $\Z_p\lb t \rb$ for every $k \ge 1$. A proper face $\tau \subset \Delta$ is called a \emph{simplex of volume~1} if it has $\dim(\tau)+1$ vertices and all integral points in the cone generated by $\tau$ are $\Z$-linear combinations of the vertex vectors. It is clear that this condition is then satisfied by all subfaces $\tau' \subset \tau$, and therefore it suffices to check it for faces of codimension~1. The latter task is quite straightforward for the simplicial and hyperoctahedral families. 

The construction of the Frobenius structure for the Picard-Fuchs operators of our two families is accomplished by the following.

\begin{proposition}\label{cartier2frobenius}
Let assumptions and notations be as in Proposition \ref{cartier-mod-exact}. Assume in addition
that the $n$-th Hasse--Witt condition holds. Then the corresponding Picard-Fuchs differential equation
$L(y)=0$ has a $p$-adic Frobenius structure over $R$ with $\sA=\sum_{j=0}^{n-1}A_j(t)\theta^j$.
\end{proposition}

\begin{proof} See~\cite[\S4]{IN} starting from~\cite[(16)]{IN}.   
\end{proof}

To determine the constants $\alpha_j=A_j(0)$ observe that $\cartier(1/f)$ is an infinite $p$-adic series with terms in $\sO_{f^\sigma}$, see \cite[(3)]{DCIII}. Then apply reduction modulo  $d \sO_{f^\sigma}$ to each term (Dwork-Griffiths reduction) and finally set $t=0$. This procedure, in a different phrasing, was more or less followed by Shapiro in~\cite{Sh09} in the particular case $n=4$ of our Theorem~\ref{main1}. In the present paper we follow an essentially different approach based on \emph{supercongruences} for expansion coefficients of rational functions. Our method relies on the result from~\cite{DCIII} which is recalled in the following paragraph. 

In \cite{DCIII} we defined for $k \ge 1$ the module of \emph{$k$-th formal derivatives} by
\be\label{fil-k-def}
\fil_k=\{\omega\in \hat\sO_f \;|\; \cartier^s(\omega)\in p^{sk}\hat\sO_{f^{\sigma^s}}\ \text{for all $s\ge1$}\}.
\ee
By $\fil_k^\sigma$ we denote the similar submodule of $\hat\sO_{f^\sigma}$. One of the main results in
\cite{DCIII} is that if the $k$-th Hasse--Witt condition
is satisfied over $R$ then
\be\label{O-f-decomposition-k}
\hat \sO_f \cong \sO_f(k)\oplus\fil_k
\ee
as $R$-modules, see~\cite[Theorem 4.2 and Corollary 5.9]{DCIII}. Here $\sO_f(k)$ is the submodule of $\sO_f$ consisting of elements of the form 
$h(\v x)/f(\v x)^m$ with $m \le k$.
Moreover, a similar decomposition holds in $\hat \sO_f^\circ$ and in $\hat \sO_{f^\sigma}$, $\hat \sO_{f^\sigma}^\circ$ respectively. Finally, let us remark that the $k$-th Hasse-Witt condition includes the $\ell$-th Hasse-Witt conditions for all $1 \le \ell < k$. See~\cite[\S 5]{DCIII} or~\cite[\S 3]{IN}. 

We now prove a lemma which will be essential for our computation of $p$-adic constants $\alpha_j$ in the following sections.

\begin{lemma}\label{fil-n-are-derivatives} 
Suppose $(\sO_f^\circ)^{\sG}/d\sO_f$ is a cyclic $\theta$-module of rank $n$. Suppose also
that the $n$-th Hasse-Witt condition holds over $\Z_p\lb t \rb$. Then we have 
\[
\fil_n \cap (\hat\sO_f^\circ)^\sG \subset d\hat\sO_f.
\]
The analogous statement holds in $\hat\sO_{f^\sigma}$.
\end{lemma}
\begin{proof} We prove the statement in $\hat\sO_f$. Take an arbitrary element $\omega \in \fil_n \cap (\hat\sO_f^\circ)^\sG$.
Using the reduction from Definition~\ref{cyclicMUM}(i) we can write
\[
\omega = \sum_{i=0}^{n-1} c_i \, \theta^i(1/f) + \sum_{i=1}^n x_i \frac{\partial}{\partial x_i} \nu_i
\]
with some coefiicients $c_0,\ldots,c_n \in R$ and some elements 
$\nu_1,\ldots,\nu_n \in \hat\sO_f$. As the $(n-1)$st Hasse-Witt condition also holds we have
decomposition~\eqref{O-f-decomposition-k} for $k=n-1$ as modules over the larger ring $\Z_p\lb t \rb$.
We write $\nu_i = \omega_i + \delta_i$ with $\omega_i \in \sO_f(n-1)$ and $\delta_i \in \fil_{n-1}$
for $i=1,\ldots,n$. Note that $x_i \frac{\partial}{\partial x_i} \delta_i \in \fil_n$ and therefore
we have 
\be\label{auxiliary-element}
\sum_{i=0}^{n-1} c_i \, \theta^i(1/f) + \sum_{i=1}^n x_i \frac{\partial}{\partial x_i} \omega_i \in \fil_n.
\ee
Let us denote the element in the left-hand side of~\eqref{auxiliary-element} by $\omega'$. Since each 
$x_i \frac{\partial}{\partial x_i} \omega_i \in \sO_f(n)$ and $c_i\theta^i(1/f)\in\sO_f(n)$
we have $\omega' \in \sO_f(n)\cap \fil_n$.
The direct sum decomposition~\eqref{O-f-decomposition-k} for $k=n$ then implies that $\omega'=0$,
and hence $\sum_{i=0}^{n-1} c_i \theta^i(1/f) \in d \sO_f$. By~\cite[Lemma 4.6]{IN} we then have $c_i=0$ for all $0 \le i \le n-1$.
We conclude that $\omega \in d\hat\sO_f$. This proves our first claim. 

In $\hat\sO_{f^\sigma}$ the same proof works using decompositions~\eqref{O-f-decomposition-k} for $\hat\sO_{f^\sigma}$. 
\end{proof}

Now we can sketch the method which is used in the subsequent sections to prove Theorems~\ref{main1} and~\ref{main2}. Let $\eta_1,\ldots,\eta_d$ be a basis of $(\sO_f^\circ)^\sG(n)$. Decomposition~\eqref{O-f-decomposition-k} holds with $k=n$ in $\hat\sO^\circ_{f^\sigma}$ over the larger ring $\Z_p\lb t \rb$.
Then there  exist $\nu_1(t),\ldots,\nu_d(t)\in \Z_p\lb t \rb$ such that
\be\label{cartier-mod-fil-n}
\cartier(1/f)\is \sum_{i=1}^d\nu_i(t)\eta_i^\sigma\mod{\fil_n^\sigma}.
\ee
By Lemma~\ref{fil-n-are-derivatives} we know that $\fil_n^\sigma \cap (\hat \sO_f^\circ)^\sG\subset d\hat\sO_f$. Hence the equality in~\eqref{cartier-mod-fil-n} also holds modulo $d\hat\sO_{f^\sigma}$, and therefore the $\nu_i(t)$ determine the $A_j(t)$ in~\eqref{cartier-defines-A-i}. Elements of $\hat\sO_f$ can be expanded as formal Laurent series $\sum_{\v u \in \Z^n}c_{\v u}(t)\v x^\v u$ with coefficients $c_\v u \in \Z_p\lb t \rb$ and the Cartier action on such formal expansions is given by $\cartier: \sum c_\v u \v x^\v u \to \sum c_{p \v u} \v x^\v u$, see~\cite[\S 3]{DCIII} and~\cite[\S 2]{DCII}. It is then clear from the definition~\eqref{fil-k-def} that the expansion of an element of $\fil_n$ or $\fil_n^\sigma$ has the property that each $c_\v u(t)$ is divisible by $p^{\ord_p(\v u)}$. Therefore~\eqref{cartier-mod-fil-n} implies congruences for the expansion coefficients of the elements $1/f$ and $\eta_1,\ldots,\eta_d$, which turn out to be sufficient to determine the unknown coefficients $\nu_i(t)$ and in particular $\nu_i(0)$.

At the end let us fulfill the promise which was made in the Introduction. We give an example of an operator $L$ with a Frobenius structure over $\Z_p\lb t \rb$ for which the respective operator $\sA$ is not unique.   

\begin{proof}[Proof of Remark~\ref{uniqueness-remark}]
Suppose $L$ is of order $n=2$ and has Frobenius structure over $\Z_p\lb t\rb$ with operator $\sA$.
Suppose also that we have a non-trivial
solution $y_0(t)\in\Z_p\lb t\rb$ and a Wronskian determinant 
$W(t):= y_0(t)(\theta y_1)(t)-(\theta y_0)(t) y_1(t)\in\Z_p\lb t\rb^\times$. Define
for any $\lambda\in\Z_p$ the operator
\[
\sA+\lambda\frac{y_0(t)}{W(t^p)}\left(y_0(t^p)\theta-\theta(y_0(t^p))\right).
\]
Then one easily verifies that this operator maps $u(t^p)$ for any solution $u$ of $L(y)=0$
to a solution of $L(y)=0$.
\end{proof}

\section{The simplicial example}\label{sec:proof-of-main1}
This section is devoted to the proof of Theorem \ref{main1}. Recall the simplicial
family given by the equation $1 - t g(\v x)=0$ with
\[
g(\v x)=x_1+\cdots+x_n+\frac{1}{x_1\cdots x_n}.
\]
This family is related to the so-called Dwork families 
$X_0^{n+1}+\cdots+X_n^{n+1}=\frac{1}{t}X_0\cdots X_n$. Simply replace $x_i$ 
in $g(\v x)=1/t$ by $X_i^{n+1}/(X_0\cdots X_n)$ for $i=1,\ldots,n$.

We work with the Laurent polynomial $f(\v x)=1-t g(\v x)$ and use the construction given in the previous section with the trivial symmetry group $\sG=\{1\}$. From \cite[Proposition 5.1]{IN} we have the following,

\begin{proposition}\label{simplicial-is-cyclic}
Over the ring $\Z[t,1/D_f(t)]$ with $D_f(t)=(n+1)(1-((n+1)t)^{n+1})$ the module $\sO_f^\circ/d\sO_f$ is a cyclic $\theta$-module of MUM-type with the Picard--Fuchs operator 
\[
L = \theta^n - ((n+1)t)^{n+1}(\theta+1)\cdots(\theta+1).
\]
\end{proposition}

Note that $L$ is the operator in Theorem~\ref{main1} and $D_f(t)=(n+1)D(t)$ where $D(t)$ is the leading coefficient of $L$.

We now fix a prime number $p>n+1$. The ring  
\[
R = \text{$p$-adic completion of } \Z[t,1/D(t)] 
\]
in our Theorem coincides with the $p$-adic completion of $\Z[t,1/D_f(t)]$ and embeds into $\Z_p\lb t \rb$.
By~\cite[Theorem 3.3]{IN} the $n$-th Hasse--Witt condition over $\Z_p\lb t \rb$ holds for the simplicial family. Then it follows from Proposition~\ref{cartier2frobenius}
that the first statement of Theorem \ref{main1} holds true. Namely, the differential operator $L$ has a Frobenius structure with $\sA=\sum_{j=0}^{n-1}A_j(t) \theta^j$ where the coefficients $A_j(t) \in R$ come from the expression
\be\label{cartier-action-on-f-mod-exact}
\cartier(1/f) \is \sum_{j=0}^{n-1} A_j(t) \, \theta^j(1/f^\sigma) \; \mod {d \hat\sO_{f^\sigma}}.
\ee
In view of Proposition~\ref{special-Frobenius-action} the respective $p$-adic constants are given by $\alpha_j=A_j(0)$. It remains to compute these values. 

In order to display the symmetry of $f(\v x)$ explicitly,
we introduce $x_0=(x_1\cdots x_n)^{-1}$. So $f=1-t(x_0+\cdots+x_n)$. 
To every $n$-tuple $\v u=(u_1,\ldots,u_n)$ of integers there exists
an $n+1$-tuple $\v U=(U_0,U_1,\ldots,U_n)$ of non-negative integers such that $\v x^{\v u}=\v x^{\v U}:=x_0^{U_0}\cdots x_n^{U_n}$. Notice that $\v U$ is determined up to shifts along the vector $(1,1,\ldots,1)$. We define $|\v U|:=U_0+\cdots+U_1$. Observe that $\v U$ is uniquely determined by $\v u$
if we require that $\min_iU_i=0$. In that case $|\v U|=\deg(\v u)$. Observe also that if $|\v U|\le n$
then at least one coefficient of $\v U$ must be zero.
For any $\v U\ge\v0$ (i.e. $U_i\ge0$ for all $i$) we define
\[
\omega_{\v U}=|\v U|!\frac{t^{|\v U|}\v x^{\v U}}{f^{|\v U|+1}}.
\]  

\begin{lemma}\label{omega-U-is-a-basis}
Suppose $p>n+1$. Then the rational functions $\omega_{\v U}$ with $|\v U|<n$ generate the $\Z_p\lb t \rb$-module generated by $\sO^\circ_f(n)$.
\end{lemma}

\begin{proof}
Since $p>n+1$ we can ignore the factorials $|\v U|!$ in $\omega_{\v U}$. It suffices to show that any
function $\frac{t^{|\v U|}\v x^{\v U}}{f^{k+1}}$ with $|\v U|\le k<n$ is a 
$\Z_p[t]$-linear combination of the $\omega_{\v U}$ with $|\v U|<n$. 
We apply induction on $k$. For $k=0$ the statement is clear because $|\v U|=0$.
Let us suppose $k\ge1$ and the statement
is true for functions $\frac{t^{|\v U|}\v x^{\v U}}{f^k}$ with $|\v U|<k$. Suppose we like to
express $\frac{t^{|\v U|}\v x^{\v U}}{f^{k+1}}$ with $|\v U|<k$ as linear combination of the
$\omega_{\v U}$.
Multiply this function by $1=(f+tg)^{k-|\v U|}$. Using the induction hypothesis it
suffices to deal with
$\frac{t^k\v x^{\v U}g(\v x)^{k-|\v U|}}{f^{k+1}}$. All monomials in the expansion of $\v x^{\v U}
g(\v x)^{k-|\v U|}$ can be written in the form $\v x^{\v K}$ with $|\v K|=k$ and hence the rational
function is a linear combination of $\omega_{\v U}$ with $|\v U|=k$. 
\end{proof}

Let us expand $\omega_{\v U}$ as a Laurent series in $x_1,\ldots,x_n$ and coefficients
in $\Z_p\pow{t}$ as follows.
\begin{eqnarray*}
|\v U|!\frac{t^{|\v U|}\v x^{\v U}}{f^{|\v U|+1}}&=&|\v U|!
\sum_{k\ge0}t^{|\v U|+k}\v x^{\v U}\binom{|\v U|+k}{k-1}
(x_0+\cdots+x_n)^k\\
&=&|\v U|!\sum_{K_0,\ldots,K_n\ge0}t^{|\v U+\v K|}
\binom{|\v U+\v K|}{|\v U|,K_0,\ldots,K_n}\v x^{\v U+\v K},
\end{eqnarray*}
where $\v K=(K_0,K_1,\ldots,K_n)$ and $\binom{\sum_ir_i}{r_0,r_1,\ldots,r_n}$
is the notation for the multinomial
coefficient $\frac{(\sum_ir_i)!}{r_0!r_1!\cdots,r_n!}$. We use the convention
that this multinomial is $0$ if
$r_i<0$ for at least one $i$. Note that the support of this Laurent series in $x_1,\ldots,x_n$
lies in $\Z^n$. However, the coefficient of $x_1^{u_1}\cdots x_n^{u_n}$ is divisible
by $t^{\deg(\v u)}$. Let us remark that the formal expansion procedure which we use here is the expansion for Calabi-Yau crystals which is discussed at the end of~\cite[\S 3]{DCIII}. The Cartier operator acts on such expansions by formula $\cartier(\sum_{\v u \in \Z^n} c_\v u(t) \v x^\v u) = \sum_{\v u \in \Z^n} c_{p \v u}(t) \v x^\v u$.

Now let $\v V\in\Z^{n+1}$ with $\min(\v V)=0$ and $N$ a positive integer. Let us compute the coefficient at $\v x^{N\v V}$ in the formal expansion of $\omega_{\v U}$.
We collect all terms in the Laurent series such that $\v x^{N\v V}=\v x^{\v U+\v K}$
(as monomials in $x_1,\ldots,x_n$) and sum their coefficients. That is,
we collect all terms with $\v K=N\v V-\v U+\ell\v 1$,
where $\ell$ is a non-negative integer and
$\v 1=(1,1,\ldots,1)$. Clearly $\ell$ must be such that $\min(N\v V-\v U+\ell\v 1)\ge0$.
 The desired coefficient becomes
\begin{equation}\label{fullcoefficient}
|\v U|!\sum_{\ell\ge 0}t^{N|\v V|+(n+1)l}\binom{N|\v V|+(n+1)\ell}
{|\v U|,\ldots,NV_i-U_i+\ell,\ldots}.
\end{equation}

\begin{lemma}\label{omega-U-in-cyclic-basis}
For any $\v U$ with $\min(\v U)=0$ and $|\v U|<n$ we have
\[
\omega_{\v U}\is \left(\prod_{i=0}^n[\theta/(n+1)]_{U_i}\right)(1/f) \quad \mod{d\sO_f}.
\]
Here $[x]_k$ is the falling factorial $[x]_k=x(x-1)\cdots(x-k+1)$.
\end{lemma}

\begin{proof}
By Proposition~\ref{simplicial-is-cyclic} we have that $\sO_f^\circ/d\sO_f$ is a cyclic $\theta$-module. Moreover, the reduction procedure in the proof of \cite[Theorem 5.1]{IN} shows that property~(i) of Definition~\ref{cyclicMUM} for $m<n$ holds with constant coefficients $b_j \in \Z[(n+1)^{-1}]$. In particular,  there exist $c_i\in\Z[(n+1)^{-1}]$ such that
\[
\omega_{\v U}\is \sum_{i=0}^{n-1}c_i\theta^i(1/f)\mod{d\sO_f}.
\]
We now determine these $c_i$. Take the constant term of the Laurent expansion on
both sides. Using~\eqref{fullcoefficient} for $\v U=\v V = \v 0$ we find that the constant term of $\theta^i(1/f)$ reads
\[
\theta^i\left(\sum_{\ell\ge0}\binom{(n+1)\ell}{\ell,\cdots,\ell}t^{(n+1)\ell}\right)=
\sum_{\ell\ge0}(n+1)^i\ell^i\binom{(n+1)\ell}{\ell,\cdots,\ell}t^{(n+1)\ell}.
\]
On the left hand side we get \eqref{fullcoefficient} for the case $\v V=\v 0$, hence
\[
|\v U|!\sum_{\ell\ge\max_iU_i}\binom{(n+1)\ell}{|\v U|,\ldots,\ell-U_i,\ldots}t^{(n+1)\ell}.
\]
This expression is easily seen to be equal to
\[
\sum_{\ell\ge0}\prod_{i=0}^n[\ell]_{U_i}\binom{(n+1)\ell}{\ell,\cdots,\ell}t^{(n+1)\ell},
\]
and therefore we must have
\[
\prod_{i=0}^n[\ell]_{U_i} = \sum_{i=0}^{n-1} c_i (n+1)^i \ell^i
\]
for all $\ell \ge 0$. Hence the two sides of this identity are equal as polynomials in $\ell$, which determines the coefficients $c_i$ uniquely and yields our claim.
\end{proof}

In Lemma \ref{omega-U-is-a-basis} we have seen that the rational functions $\omega_\v U$ with $|\v U| < n$ generate the 
module $\sO_f^\circ(n)$. We now switch to another set of generators.
Let $S(m,k)$ for integers $m\ge k\ge0$ be the Stirling
numbers of the second kind. They are defined by the property that 
\[
x^m=\sum_{k=0}^mS(m,k)[x]_k
\]
for all $m,k\ge0$ where $[x]_k$ is the falling factorial $x(x-1)\cdots(x-k+1)$.
Note that $S(m,m)=1$ for $m \ge 0$
and $S(m,0)=0$ if $m\ge1$.

\begin{lemma}\label{eta-basis} For an $n+1$-tuple $\v U$ with $|\v U|<n$ we define 
\[
\eta_{\v U}=\sum_{K_0=0}^{U_0}\cdots\sum_{K_n=0}^{U_n}S(U_0,K_0)\cdots S(U_n,K_n)\omega_{(K_0,\ldots,K_n)}.
\]
Then,
\begin{itemize}
\item[(i)] one has $\eta_\v U = (\theta/(n+1))^{|\v U|}(1/f) \; \mod {d \sO_f}$,

\item[(ii)] the functions $\eta_\v U$ with $|\v U|<n$ generate the $\Z_p\lb t\rb$-module generated by $\sO_f^\circ(n)$,

\item[(iii)] there are power series $\lambda_\v U \in \Z_p\lb t \rb$ for $|\v U|<n$ such that
\be\label{cartier-on-f-mod-fil-n}
\cartier(1/f) = \sum_{\v U: |\v U|<n} \lambda_\v U(t) \eta_\v U^\sigma \; \mod {p^n \fil_n^\sigma}.
\ee
\end{itemize}
Moreover, the coefficients $A_i(t)$, $0 \le i < n$ in~\eqref{cartier-action-on-f-mod-exact}
are given by
\[
A_i(t) = p^{-i}(n+1)^{-i} \sum_{\v U: |\v U|=i} \lambda_\v U(t). 
\]
\end{lemma}
\begin{proof} (i) follows from Lemma~\ref{omega-U-in-cyclic-basis}.

(ii) follows from Lemma~\ref{omega-U-is-a-basis} and the observation that $\eta_\v U = \omega_\v U + $ a linear combination of 
$\omega_\v K$ with $|\v K|<|\v U|$.

Recall that the $n$th Hasse-Witt condition holds for the simplicial family over the ring $\Z_p\lb t \rb$. It was explained in Section~\ref{sec:Frobenius-structure} that over this larger ring we then have the decomposition
$\hat\sO_{f}^\circ = \sO_{f}^\circ(n) \oplus \fil_n$. Moreover, by~\cite[Corollary 4.3]{DCIII} we also have that 
\[
\cartier(\sO_f^\circ) \subset \sO_{f^\sigma}^\circ(n) + p^n \; \fil_n^\sigma .
\]
Existence of the $\lambda_{\v U}$'s in (iii) follows from this and (ii). 

By Lemma~\ref{fil-n-are-derivatives} applied in
$\hat\sO_{f^\sigma}$ we have $\fil_n^\sigma \cap \hat\sO_{f^\sigma}^\circ \subset d \hat\sO_{f^\sigma}$.
Hence identity~\eqref{cartier-on-f-mod-fil-n} is valid 
modulo~$d \hat\sO_{f^\sigma}$ and using~(i) we get
\[
\cartier(1/f) = \sum_{\v U: |\v U|<n} \lambda_\v U(t) p^{-|\v U|}(n+1)^{-|\v U|}\theta^{|\v U|}(1/f^\sigma)
\mod {d\hat\sO_{f^\sigma}}.
\]
Combined with Proposition \ref{cartier-mod-exact} and the independence of the $\theta^i(1/f^\sigma)$
modulo $d\hat\sO_{f^\sigma}$ (see \cite[Lemma 4.8]{IN}) we obtain the desired expression for the $A_i(t)$.
\end{proof}

We now proceed to evaluate $A_i(0)$.

\begin{lemma}\label{limit2}
Let $\v V$ be an $n+1$-tuple of integers with $\min(\v V)=0$ and let $N$ be a large
positive integer. Then the limit of $t^{-N|\v V|}$ times the coefficient of $\v x^{N\v V}$ in
$\eta_{\v U}$ as $t\to0$ equals
\[
\binom{N|\v V|}{NV_0,NV_1,\ldots,NV_n}\prod_{i=0}^n (NV_i)^{U_i}.
\]
Here $V_i^{U_i}$ should be read as $1$ when $U_i=V_i=0$.
\end{lemma}
\begin{proof}  The coefficient of $\v x^{N\v V}$ in $\omega_\v U$ was computed 
in~\eqref{fullcoefficient}. After division by $t^{N|\v V|}$ this expression equals 
\begin{equation}\label{dividedcoefficient}
|\v U|!\sum_{\ell \ge 0}t^{(n+1)l}\binom{N|\v V|+(n+1)\ell}{|\v U|,\ldots,NV_i-U_i+\ell,\ldots}.
\end{equation}
We set $t=0$. In the summation the multinomial coefficient vanishes unless 
$\min(N\v V-\v U+\ell\v 1)\ge0$. In particular, when $V_i=0$ and $U_i>0$ for some $i$
this implies that $\ell>0$
and~\eqref{dividedcoefficient} vanishes at $t=0$.
In all other cases we can start with $\ell=0$ and get
\[
|\v U|!\binom{N|\v V|}{|\v U|,\ldots,NV_i-U_i,\ldots}=\binom{N|\v V|}{NV_0,NV_1,\ldots,NV_n}
\prod_{i=0}^n NV_i(NV_i-1)\cdots(NV_i-U_i+1).
\]
Notice that the latter product vanishes if $V_i=0$ and $U_i>0$ for some $i$, so this is the result for all cases. Using the expression for $\eta_\v U$ from Lemma~\ref{eta-basis} we find that the limit at $t=0$ of $t^{-N|\v V|}$ times the coefficient of $\v x^{N \v V}$ in $\eta_\v U$ equals
\[\bal
\sum_{k_0=0}^{U_0}\cdots\sum_{k_n=0}^{U_n}S(U_0,k_0)\cdots S(U_n,k_n) \binom{n |\v V|}{NV_0,\ldots,NV_n}\prod_{i=0}^n[N V_i]_{k_i} \\
= \binom{n |\v V|}{N V_0,\ldots,NV_n}\prod_{i=0}^n S(U_i,k_i)[N V_i]_{k_i} = \binom{N |\v V|}{N V_0,\ldots,NV_n}\prod_{i=0}^n (N V_i)^{U_i},
\eal\]
which is our desired expression.
\end{proof}

\begin{proposition}\label{lambda-U-via-gamma-p}
For any $n+1$-tuple of integers $\v V$  with $|\v V|\le n$ one has
\begin{equation}\label{expansion1}
\frac{\Gamma_p(p^{s+1}|\v V|)}{\prod_{i=0}^n\Gamma_p(p^{s+1}V_i)}\is
\sum_{\v U:|\v U|<n}\lambda_{\v U}(0)p^{s|\v U|}
\prod_{i=0}^n(V_i)^{U_i} \quad \mod{p^{(s+1)n}}
\end{equation}
for all integers $s \ge 0$.
\end{proposition}

\begin{proof} We shall take the coefficient of $\v x^{p^s \v V}$ in~\eqref{cartier-on-f-mod-fil-n}, divide on both sides by $t^{p^{s+1}|\v V|}$ and set $t=0$. In~\cite[\S3]{DCIII} we explain that the Cartier operator acts on formal expansions as \[
 \cartier \left(\sum_{\v u \in \Z_n} c_\v u \v x^\v u \right)=\sum_{\v u \in \Z_n} c_{p\v u} \v x^\v u
\]
and the expansion coefficient at $\v x^{p^s \v V}$ of an element of $\fil_n$ is divisible by $p^{sn}$. Using Lemma~\ref{limit2} with $N=p^{s+1}$ and $N=p^s$ we obtain
\[
\binom{p^{s+1}|\v V|}{p^{s+1}V_0,\ldots,p^{s+1}V_n}\is \sum_{\v U: |\v U|<n}\lambda_{\v U}(0)
\binom{p^s|\v V|}{p^sV_0,\ldots,p^sV_n}\prod_{i=0}^n (p^sV_i)^{U_i}\mod{p^{(s+1)n}}.
\]
Since $|\v V|\le n<p$, it follows that the multinomial coefficient $\binom{p^s|\v V|}{p^sV_0,\ldots,p^sV_n}$ is a $p$-adic unit. Dividing by this unit on both sides we obtain~\eqref{expansion1}. 
\end{proof}

\begin{proposition}
For any $\v V$ consider the Taylor series expansion
\[
\frac{\Gamma_p(x|\v V|)}{\prod_{i=0}^n\Gamma_p(x V_i)}= \sum_{k \ge 0} a_k(\v V) x^k.
\] 
Then, for any $\v V$ with $|\v V|\le n$, one has
\be\label{equations-for-lambda-U-at-0}
\sum_{\v U:|\v U|= k}\lambda_{\v U}(0) 
\prod_{i=0}^n(V_i)^{U_i} = p^k a_k(\v V) 
\ee
for $k=0,1,\ldots,n-1$.
\end{proposition}

\begin{proof}
From the properties of Morita's $p$-adic $\Gamma$-function it follows that $a_k(\v V)\in\Z_p$
for all $k\ge0$. Setting $x=p^{s+1}$ for any $s>0$ we get the $p$-adically converging series
\[
\frac{\Gamma_p(p^{s+1}|\v V|)}{\prod_{i=0}^n\Gamma_p(p^{s+1} V_i)}= \sum_{k \ge 0} a_k(\v V) p^kp^{sk}.
\] 
Consider this equality modulo $p^{(s+1)n}$ and compare with \eqref{expansion1}.
Our Proposition follows by letting $s\to\infty$ and induction on $k$. 
\end{proof}

\begin{proof}[Proof of Theorem~\ref{main1}.] By Proposition~\ref{cartier2frobenius}
there exists a Frobenius structure for the differential operator $L$ given by $\sA=\sum_{k=0}^{n-1}A_k(t) \theta^k$ with the coefficients $A_k(t) \in R \subset \Z_p\lb t \rb$. By Proposition~\ref{special-Frobenius-action} we have $\alpha_k = A_k(0)$ and by Lemma~\ref{eta-basis}(iii) we have $A_k(0)= p^{-k}(n+1)^{-k} \sum_{\v U: |\v U|=k} \lambda_\v U(0)$.

The latter sum could be computed easily if we could set $V_0=V_1=\cdots=V_n=1$ in 
\eqref{equations-for-lambda-U-at-0}. Unfortunately the latter equation only holds for $\v V$ with $|\v V|\le n$.
In particular $\min_iV_i=0$ for these vectors $\v V$. We use a small detour to arrive at our desired conclusion.

It turns out that $(-1)^{|\v V|}$ times the left-hand side of \eqref{equations-for-lambda-U-at-0} summed
over all $\v V\in\{0,1\}^{n+1}$ is also $0$. To see this consider $(-1)^{|\v V|}$ times the summand
$\v V^{\v U}:=\prod_jV_j^{U_j}$. Since $|\v U|\le n$ there exist $i$ such that $U_i=0$. To any 
$\v V\in\{0,1\}^{n+1}$ there exists $\tilde{\v V}$ obtained by replacing $V_i$ in $1-V_i$. Clearly
$|\v V|$ and $|\tilde{\v V}|$ differ by $1$ and $V_i^{U_i}=V_i^0=1$ independent of $V_i$. Hence the
contributions corresponding to $\v V$ and $\tilde{\v V}$ cancel. This proves that the alternating sum
of the left hand sides of \eqref{equations-for-lambda-U-at-0} is $0$. 

It also turns out that $(-1)^{|\v V|}$ times $a_k(\v V)$ summed over $\v V\in\{0,1\}^{n+1}$ is $0$. 
To see this consider the sum
\[
\sum_{\v V\in\{0,1\}^{n+1}}(-1)^{|\v V|}\frac{\Gamma_p(x|\v V|)}{\prod_{i=0}^n\Gamma_p(xV_i)}
=\sum_{j=0}^{n+1}(-1)^{j}\binom{n+1}{j}\frac{\Gamma_p(j x)}{\Gamma_p(x)^j}.
\]
The Taylor series expansion of this function is divisible by $x^{n+1}$. This is because for any power
series $F(x)=1+f_1 x + f_2 x^2 + \ldots$ one has 
\be\label{formal-series-identity}
\sum_{j=0}^{n+1} (-1)^j \binom{n+1}{j} \frac{F(j x)}{F(x)^j} \is 0 \; \mod {\sO(x^{n+1})}.
\ee
To prove this identity we introduce an auxiliary variable $y$ and observe that the coefficient of $x^k$ in 
$\sum_{j=0}^{n+1} (-1)^j \binom{n+1}{j} y^j F(j x)$ is equal to 
$f_k \sum_{j=0}^{n+1} (-1)^j \binom{n+1}{j} y^j j^k = (y \frac{d}{dy})^k(1-y)^{n+1}$. When $k<n+1$ this latter expression is divisible by $(1-y)^{n+1-k}$. Substituting $y=1/F(x)=1-f_1x+\ldots$ we have that $1-y$ is divisible by $x$ and hence the whole sum is divisible by $x^{k+n+1-k}=x^{n+1}$.

We have now seen two alternating series that add up to $0$. From \eqref{equations-for-lambda-U-at-0}
we see that for all $\v V$ with $|\v V|\le n$ we have equality of terms in these summations.
Hence we also have equality for the remaining term with $\v V=(1,1,\ldots,1)$. 
\end{proof}

\section{The hyperoctahedral case}\label{sec:proof-of-main2}
This section is devoted to the proof of Theorem~\ref{main2}. Recall that the hyperoctahedral family is given by the equation $1-t g(\v x)=0$ with
\[
g(\v x)=x_1+\frac{1}{x_1}+\cdots+x_n+\frac{1}{x_n}.
\]
Let $\sG \cong S_n \times (\Z/2\Z)^n$ be the finite group generated by the
permutations of the $x_i$ and the inversions $x_i\to x_i^{\pm1}$. The polynomial $g(\v x)$ is invariant under $\sG$.

We work with the polynomial $f(\v x)=1-t g(\v x)$ and use the construction from Section~\ref{sec:Frobenius-structure}. From \cite[Proposition 6.1]{IN} we have the following,
\begin{proposition}\label{hyperoctahedral-is-cyclic}
There exists 
$D_f(t)\in n!\Z[t]$ with $D_f(0)=n!$ such that over the ring
$\Z[t,1/D_f(t)]$ the module $M=(\sO_f^\circ)^{\sG}/d\sO_f$ satisfies property (i)
of Definition~\ref{cyclicMUM}. Moreover, there exists a differential operator $L\in \Z[t][\theta]$ of order $n$ 
with leading coefficient $D(t)=D_f(t)/n!$ and of MUM-type such that $L(1/f) \in d\sO_f$.
If $L$ is irreducible in $\Q(t)[\theta]$, then property (ii) of Definition~\ref{cyclicMUM} is
satisfied and $M$ is a cyclic $\theta$-module of MUM type.  
\end{proposition}

The differential operator in the statement of~\cite[Proposition 6.1]{IN} is monic with coefficients in $\Z[t,1/D_f(t)]$. Here we multiply it on the left by the polynomial $D(t)$. The fact that our $L$ has coefficients in $\Z[t]$ follows from the proof of~\cite[Proposition 6.1]{IN}, in the notation of which we have $L=\sum_{i=0}^n a_i(t) \theta^i$ and $D(t)=a_0(t)$. This $L$ is precisely the operator in Theorem~\ref{main2}.

We suppose from now on that the operator $L$ is irreducible and fix a prime number $p > n$.
Existence of the Frobenius structure $\sA=\sum_{j=0}^{n-1} A_j(t)\theta^j$ with coefficients in the ring $R =$ $p$-adic completion of $\Z[t,1/D(t)]$
 follows from Proposition~\ref{cartier2frobenius}
and the $n$-th Hasse--Witt
condition, which in the hyperoctahedral case is satisfied over $\Z_p\lb t \rb$ by \cite[Theorem 3.3]{IN}.
It remains to compute the $p$-adic values  $\alpha_i=A_i(0)$.

For $\v u \in \Z^n$ we denote 
\[
(\v x^\v u)^{\sG} = \# \sG^{-1} \sum_{h \in \sG} \v x^{h(\v u)}.
\]
We define elements
\[
\omega_\v u = \deg(\v u)! \frac{t^{\deg(\v u)}(\v x^{\v u})^{\sG}}{f(\v x)^{\deg(\v u)+1}}. 
\]

Note that $\omega_{\v u}=\omega_{h(\v u)}$ for any $h\in \sG$. 
The elements $\omega_\v u$ with $\v u\ge\v 0$
and $\deg(\v u)<n$ generate $(\sO_f^\circ)^{\sG}(n)$.
Here $\v u\ge\v 0$ means $u_i\ge0$ for $i=1,\ldots,n$. For such non-negative vectors we have $\deg(\v u)=|\v u|:=u_1+\cdots+u_n$.
Similarly to the simplicial case we define for all $\v u\ge\v0$, $\deg(\v u)<n$,
\[
\eta_\v u = \sum_{k_1=0}^{u_1} \ldots \sum_{k_n=0}^{u_n} S(u_1,k_1)\ldots S(u_n,k_n) \omega_{(k_1,\ldots,k_n)}.
\]  
Here $S(m,k)$ are the Stirling numbers of the second kind. 
These elements also generate $(\sO_f^\circ)^{\sG}(n)$.
Using the reduction procedure from the proof of \cite[Proposition 6.1]{IN}  one can write
\be\label{eta-to-cyclic-basis}
\eta_{\v u} = \sum_{j=0}^{|\v u|} \mu_{\v u,j}(t) \theta^j(1/f) \mod {d \sO_f}
\ee
with some $\mu_{\v u,j}\in\Z_p[t]$. 

\begin{lemma}\label{eta-to-cyclic-basis-at-0}
Let $\v u\ge\v 0$ and $0 \le j \le |\v u|<n$. Then
\[
\mu_{\v u,j}(0) = \begin{cases} \dfrac1{2^{j}}\dfrac{(n-j)!}{n!}\quad \text{ if }\; 
\v u = (\underbrace{1,\ldots,1}_{j},0,\ldots,0)\text{ or a permutation of it}\\
0, \text{ otherwise }.
\end{cases}
\]
\end{lemma}

\begin{proof} 
Let us compute the constant term in the formal Laurent series expansion of $\omega_\v u$.
This expansion is given by
\[
\omega_\v u = |\v u|! \sum_{\ell=0}^\infty t^{\ell+|\v u|} 
\binom{\ell+|\v u|}{|\v u|} g(\v x)^\ell (\v x^{\v u})^{\sG}.
\]  
Since $g(\v x)$ is $\sG$-invariant, the constant term of 
$g(\v x)^\ell \v x^{h(\v u)}$ is the same as the constant term of 
$g(\v x)^\ell \v x^{\v u}$ for any $h\in \sG$. It vanishes unless $\ell=2k+|\v u|$
for an integer $k \ge 0$, and in the latter case we have
\[
\text{ const. term of } g(\v x)^\ell \v x^\v u = 
\sum_{k \in \Z_{\ge 0},\; k_1+\cdots+k_n=k}\frac{(2k+2|\v u|)!}{k_1!(k_1+u_1)!
\cdots k_n!(k_n+u_n)!}.
\]
Therefore,
\[\bal
\text{ const. term of } \omega_\v u =& \sum_{\v k \ge\v 0} t^{2|\v k|+2|\v u|}
\frac{(2k+2|\v u|)!}{k_1!(k_1+u_1)!\cdots k_n!(k_n+u_n)!} \\
=& \sum_{\v m \ge 0} t^{2|\v m|}\frac{(2|\v m|)!}{(m_1! \cdots m_n!)^2}
\prod_{i=1}^n [m_i]_{u_i},
\eal\]
where we denoted $\v k=(k_1,\ldots,k_n)$ and
$\v m=(k_1+u_1,\ldots,k_n+u_n)$. The summands where some $m_i<u_i$ vanish because the
respective $[m_i]_{u_i}=0$. For $\v u\ge\v 0$
we denote by $F_{\v u}(t)$ the constant term of
$\eta_\v u$. It follows from the above that 

\be\label{constant-terms-hyperoctahedral}
F_{\v u}(t) =  \sum_{\v m \ge 0} t^{2|\v m|} \frac{(2|\v m|)!}{(m_1! \cdots m_n!)^2}
\prod_{i=1}^n m_i^{u_i}.
\ee

Note that $\eta_{\v 0}=\omega_{\v 0}=1/f$. We denote $F_{\v 0}(t)$ by $F(t)$, 
this power series is a solution of the differential equation $L(F)=0$.
As the constant terms of elements in $d \sO_f$ vanish, we should have
\be\label{mu-via-F}
F_{\v u}(t) = \sum_{j=0}^{|\v u|} \mu_{\v u,j}(t) \, (\theta^j F)(t)
\ee
for every $\v u\ge\v 0$ with $|\v u|<n$. Moreover, since $L$ is irreducible,
the polynomials $\mu_{\v u,j}(t)$ are uniquely determined by the relations~\eqref{mu-via-F}. 

We will now show by induction in $|\v u|$ that all $\mu_{\v u,j}(t) \in \Z_p[t^2]$ and
$\mu_{\v u,j}(t)$ is divisible by $t^2$ unless $\max_iu_i=1$ and $j=|\v u|$. 
When $\v u = \v 0$ we trivially have $\mu_{\v 0,0}=1$. When $|\v u|=1$ the only possibility is that
$\v u$ is a coordinate permutation of $(1,0,\ldots)$. We check that
\[
(\theta F)(t) = 2 \sum_{\v m \ge\v 0} t^{2|\v m|} 
\frac{(2|\v m|)!}{(m_1! \cdots m_n!)^2} (m_1+\cdots+m_n) = 2n F_{(1,0,\ldots)}(t),
\]
so $\mu_{\v u,0}(t)=0$ and $\mu_{\v u,1}(t)=1/2n$. Now take a vector $\v u$ with $|\v u|>1$ and assume that the claim is proved for all vectors with smaller sum of coordinates. If 
$\max_iu_i>1$ we have, assuming $u_1=\max_iu_i$, 
\[
\frac{(2m)!}{(m_1! \cdots m_n!)^2} \prod_{i=1}^n m_i^{u_i} = 2m(2m-1) \frac{(2m-2)!}{((m_1-1)!m_2! \cdots m_n!)^2}m_1^{u_1-2}\prod_{i=2}^n m_i^{u_i}.
\]
Writing $m_1^{u_1-2}=\sum_{k=0}^{u_1-2}\binom{u_1-2}k(m_1-1)^k$, we see that
\[
F_\v u(t) = \theta(\theta-1)t^2 \sum_{k=0}^{u_1-2}\binom{u_1-2}{k} F_{(k,u_2,\ldots,u_n)}(t),
\]
and our claim for $\v u$ follows from the induction hypothesis. 

It remains to consider the case when $\v u$ is a coordinate permutation of 
$(\underbrace{1,\ldots,1}_{\ell},0,\ldots,0)$
where $\ell=|\v u|$. We then compute that
\[\bal
(\theta^\ell F)(t) &= 2^\ell \sum_{\v m \ge\v 0} t^{2|\v m|} \frac{(2|\v m|)!}{(m_1! \cdots m_n!)^2}
(m_1+\cdots+m_n)^\ell \\
&= 2^\ell \frac{n!}{(n-\ell)!} F_{\v u}(t) \\
&\qquad + \text{a $\Z$-linear combination of } F_{\v w}(t) \text{ with } |\v w|=\ell \text{ and } 
\max_iw_i>1.
\eal\]
We proved above that each such $F_{\v w}(t)$ is a linear combination of $\theta^j F$ with $j \le \ell$
and coefficients $\mu_{\v w,j}(t)$ divisible by $t^2$. It follows that $\mu_{\v u,j}(t)$ is divisible
by $t^2$ for all $j<\ell$ and $\mu_{\v u,\ell}(t)\in 2^{-\ell} (n-\ell)!/n! + t^2 \Z_p[t^2]$. 
This finishes the proof of our induction step and also proves the formula claimed in the lemma. 
\end{proof}

Recall that the $n$-th Hasse-Witt condition holds over the larger ring $\Z_p\lb t \rb$. As we explained in  Section~\ref{sec:Frobenius-structure}, the main result of~\cite{DCIII} implies that over this larger ring we have the decomposition $\hat \sO_f^\circ \cong \sO_f^\circ(n) \oplus \fil_n$. Moreover, by~\cite[Corollary 4.3]{DCIII} one has $\cartier(\sO_f^\circ) \subset \sO_{f^\sigma}^\circ(n) + p^n \fil_n^\sigma$. Since the Cartier operator commutes with monomial substitutions we also conclude that $\cartier$ maps $(\sO_f^\circ)^\sG$ into $(\sO_{f^\sigma}^\circ)^{\sG}(n) + p^n \fil_n^\sigma$. Since the elements $\eta_\v u^\sigma$ with $|\v u|<n$ generate $(\sO_{f^\sigma}^\circ)^{\sG}(n)$ we can write
\be\label{cartier-on-f-to-eta-basis-hyperoctahedral}
\cartier(1/f) = \sum_{\v u\ge\v 0: |\v u| < n}
\lambda_\v u(t) \; \eta_\v u^\sigma \; \mod {p^n \fil_n^{\sigma}}
\ee
with some coefficients $\lambda_\v u(t) \in \Z_p\lb t \rb$. Note that the $\lambda_{\v u}(t)$ are not necessarily uniquely determined. 
For our proofs this will not matter, we simply make a choice.

\begin{lemma}\label{lambda-via-gamma-expansion-lemma} For any 
$\v u\ge\v 0$ we denote $\ell(\v u)=\#\{i: u_i>0\}$, the size of the support of $\v u$. 
For every $0 \le m \le n$ and sufficiently large
integers $s$ one has
\be\label{lambda-via-gamma-mod-ps}
\frac{\Gamma_p(p^{s+1}m)}{\Gamma_p(p^{s+1})^{m}} 
\is \sum_{\v u\ge\v0: |\v u|< n} \lambda_\v u(0) 
\frac{\binom{m}{\ell(\v u)}}{2^{\ell(\v u)}\binom{n}{\ell(\v u)}} p^{s |\v u|}
\quad \mod {p^{(s+1)n}}.
\ee
\end{lemma}
Notice that the sum on the right gets its non-zero contributions only from 
terms with $\ell(\v u)\le m$.
\begin{proof} Take some integral vectors $\v u, \v v\ge\v 0$.
For a large integer $N$ we take the expansion coefficient at $\v x^{N \v v}$ in $\omega_\v u$,
which is denoted by $[\omega_\v u]_{\v x^{N \v v}}$. We divide this coefficient by $t^{N |\v v|}$
and evaluate at $t=0$. Consider the expansion 
\[
\omega_\v u = |\v u|! \sum_{k=0}^\infty t^{k+|\v u|} 
\binom{k+|\v u|}{|\v u|} g(\v x)^k (\v x^{\v u})^{\sG}
\]  
and look at the coefficient at $\v x^{N \v v}$ in $g(\v x)^k \v x^{h(\v u)}$ for some $h \in G$.
Write $N \v v = \v w + h(\v u)$ for a vector $\v w$ in the support of $g(\v x)^k$. 
The monomial $\v x^{\v w + h(\v u)}$ in $g(\v x)^k x^{h(\v u)}$ can contribute
to the value of $t^{-N |\v v|}[\omega_\v u]_{\v x^{N \v v}}$ at $t=0$ 
if and only if
$k+|\v u|=N|\v v|$. Together with $N \v v = \v w + h(\v u)$ this implies that 
$\v w,h(\v u)\ge\v 0$. 
We conclude that $h(\v u) = \sigma(\v u)$ for some permutation $\sigma \in S_n$.
Conversely, for any $\sigma \in S_n$ such that all coordinates of $\v w=N \v v - \sigma(\v u)$
are non-negative we should take the contribution from the monomial $\v x^\v w$ in $g(\v x)^k$ with
$k=N|\v v|-|\v u|$. Recall that $\sG=S_n \times (\Z/2\Z)^n$. An element of $(\Z/2\Z)^n$
either fixes $\v u$ or moves it outside of $\Z_{\ge0}^n$. 
The number of elements $\varepsilon \in (\Z/2\Z)^n$ such that $\varepsilon(\v u)=\v u$
equals $2^{n-\ell(\v u)}$, and with every permutation $\sigma$ as above we should
account for contributions from all $h \in \sG$ of the form $h=\sigma \varepsilon$. Therefore
\[\bal
t^{-N \v v} [\omega_\v u]_{\v x^{N \v v}} \Big|_{t=0} &= \frac1{\#\sG} \sum_{\sigma \in S_n} 
\frac{(N|\v v|)!}{(N|\v v|-|\v u|)!} \frac{(N|\v v|-|\v u|)!}{(N v_1 - u_{\sigma(1)})!
\cdot (N v_n - u_{\sigma(n)})!} 2^{n - \ell(\v u)}\\
&= \frac{2^{n - \ell(\v u)}}{n! \cdot 2^n} \frac{(N|\v v|)!}{(Nv_1)!\cdots (Nv_n)!} 
\sum_{\sigma \in S_n}  \prod_{i=1}^n [N v_i]_{u_{\sigma(i)}}.\\
\eal\]    
Here $[x]_k$ are the falling factorials, and we have $[N v_i]_{u_{\sigma(i)}}=0$ when $N v_i < u_{\sigma(i)}$.
Particularly, for
\[
\v v = (\underbrace{1,\ldots,1}_{m},0,\ldots,0)
\]
we get
\[\bal
t^{-N \v v} [\omega_\v u]_{\v x^{N \v v}} \Big|_{t=0} &= \begin{cases} 0, \quad \ell(\v u)>m \\
\frac{1}{2^{\ell(\v u)} n!} \frac{(Nm)!}{N!^m} \left(\prod_{i=1}^n [N]_{u_i}\right) \binom{m}{\ell(\v u)} \ell(\v u)! (n-\ell(\v u))!
\end{cases} \\
&= \frac{1}{2^{\ell(\v u)}} \frac{(Nm)!}{N!^m} \frac{\binom{m}{\ell(\v u)}}{\binom{n}{\ell(\v u)}}\left(\prod_{i=1}^n [N]_{u_i}\right)
\eal\]
Switching to $\eta_\v u$, we get $t^{-N \v v} [\eta_\v u]_{\v x^{N \v v}}|_{t=0} = 2^{-\ell(\v u)} \frac{(Nm)!}{N!^m}  N^{|\v u|} \binom{m}{\ell(\v u)}/\binom{n}{\ell(\v u)}$. We now take the coefficient of $\v x^{p^s \v v}$ in~\eqref{cartier-on-f-to-eta-basis-hyperoctahedral}, divide by $t^{p^{s+1}m}$ and evaluate at $t=0$. Recall that the Cartier operator acts on formal expansions by $\cartier(\sum_{\v u \in \Z^n} c_\v u(t) \v x^\v u) = \sum_{\v u \in \Z^n} c_{p \v u}(t) \v x^\v u$ and the coefficient at $\v x^{p^s \v v}$ of an element of $\fil_n^\sigma$ is divisible by $p^{sn}$. Therefore we obtain the congruence
\[
\frac{(p^{s+1}m)!}{(p^{s+1}!)^m} \is \frac{(p^{s}m)!}{(p^{s}!)^m} 
\sum_{\v u\ge\v 0: |\v u|< n} 
\lambda_\v u(0) \frac{\binom{m}{\ell(\v u)}}{2^{\ell(\v u)}\binom{n}{\ell(\v u)}} p^{s |\v u|} 
\quad \mod {p^{(s+1)n}}.
\] 
Our claim follows after division by the $p$-adic unit $(p^{s}m)!/(p^{s}!)^m$.
\end{proof}

\begin{corollary}\label{lambda-via-gamma-expansion}
With the same notations as in Lemma \ref{lambda-via-gamma-expansion-lemma} we have for $j=0,1,\ldots,n-1$,
\[
\text{ coeff. of } x^j \text { in } \frac{\Gamma_p(m x)}{\Gamma_p(x)^m} =  
p^{-j}\sum_{\v u\ge\v 0: |\v u| = j} 
\lambda_\v u(0) \frac{\binom{m}{\ell(\v u)}}{2^{\ell(\v u)}\binom{n}{\ell(\v u)}}.
\] 
\end{corollary}

\begin{proof}
Since $s$ in~\eqref{lambda-via-gamma-mod-ps} can be arbitrarily large, the statement follows 
by induction on $j$.
\end{proof}

\begin{proof}[Proof of Theorem~\ref{main2}]
By Proposition~\ref{cartier2frobenius}
there exists a Frobenius structure for the differential operator $L$ with $\sA=\sum_{j=0}^{n-1} A_j(t) \theta^j$, where the coefficients $A_j \in R$ come from the expression
\be\label{cartier-mod-exact-1}
\cartier(1/f) = \sum_{j=0}^{n-1} A_j(t) \theta^j(1/f^\sigma) \mod {d \hat\sO_{f^\sigma}}.
\ee
By Lemma~\ref{fil-n-are-derivatives} we have 
\[
\fil_n^{\sigma} \cap (\hat\sO_{f^{\sigma}}^\circ)^{\sG} \subset d\hat\sO_{f^{\sigma}}.
\]
This implies that the congruence ~\eqref{cartier-on-f-to-eta-basis-hyperoctahedral} also holds
modulo $d\hat\sO_{f^{\sigma}}$. Then, using \eqref{eta-to-cyclic-basis} we get
\[
\cartier(1/f)\is \sum_{|\v u|<n}\sum_{j=0}^{n-1}\lambda_{\v u}(t)
\mu_{\v u,j}(t^p)p^{-j}\theta^j(1/f^\sigma)\mod{d\hat\sO_{f^\sigma}}.
\]
Hence the coefficients $A_j(t)$ from \eqref{cartier-mod-exact-1} read
\[
A_j(t)=p^{-j}\sum_{|\v u|<n}\lambda_{\v u}(t)\mu_{\v u,j}(t^p).
\]
Set $t=0$ and use Lemma~\ref{eta-to-cyclic-basis-at-0}. We obtain
\be\label{lambda-j-sum}
A_j(0) = p^{-j}\dfrac1{2^{j}}\dfrac{(n-j)!}{n!}
\sum_{|\v u|=\ell(\v u)=j}\lambda_{\v u}(0)
\ee
Multiply the equality in Corollary \ref{lambda-via-gamma-expansion} by $(-1)^{j-m}\binom{j}{m}$
and sum over $m=0,\ldots,j$. It is straightforward to verify that
\[
\sum_{m=0}^j(-1)^{j-m}\binom{j}{m}\binom{m}{\ell(\v u)}
\]
is $1$ if $j=\ell(\v u)$ and $0$ for all other $j$. Hence we obtain
\[
\text{ coeff. of } x^j \text { in }\sum_{m=0}^j(-1)^{j-m}\binom{j}{m}
\frac{\Gamma_p(m x)}{\Gamma_p(x)^m}
=p^{-j}\sum_{|\v u| = \ell(\v u)=j} 
\frac{1}{2^{j}\binom{n}{j}}\lambda_\v u(0) .
\]
Using \eqref{lambda-j-sum} we see that the last sum equals $j!A_j(0)$. Our Theorem now follows from the following lemma.
\end{proof}

\begin{lemma}
For any series $F(x)=1+a_1 x + a_2 x^2+\ldots$ and $j \ge 1$ one has
\begin{equation}\label{simp}
\frac1{j!}\sum_{m=0}^j(-1)^{j-m}\binom{j}{m}
\frac{F(m x)}{F(x)^m} = c x^j + O(x^{j+1})
\end{equation}
where $c$ is the coefficient of $x^j$ in $F(x)e^{- a_1 x}$.    
\end{lemma}
\begin{proof}
One can rewrite the right-hand side in~\eqref{simp} as 
\[\bal
\sum_{k \ge 0} a_k x^k & \left(\frac1{j!}\sum_{m=0}^j(-1)^{j-m}  \binom{j}{m}  \frac{m^k}{F(x)^m} \right) \\
&= \sum_{k \ge 0} a_k x^k \frac{(-1)^j}{j!} \left( (Y \frac{d}{dY})^k (1-Y)^j \right)\Big|_{Y=1/F(x)=1-a_1 x + \ldots}.
\eal\]
For $k \le j$ the expression $(Y \frac{d}{dY})^k (1-Y)^j$ is a linear combination of terms $(1-Y)^{j-i}Y^i$ with $i=0,\ldots,k$, and after substitution we have $(1-Y)^{j-i}Y^i\Big|_{Y=1/F(x)} = O(x^{j-i})$. Thus the term with $i=k$ gives the lowest degree contribution. Since this term enters the linear combination with the coefficient $(-1)^k\frac{j!}{(j-k)!}$ and $1-Y=a_1 x+O(x^2)$, and we find that our claim is true with
\[
c \; = \; \sum_{k=0}^{j} a_k  \frac{(-1)^j}{j!} (-1)^k\frac{j!}{(j-k)!} a_1^{j-k}  = \sum_{k=0}^{j} a_k  \frac{(-a_1)^{j-k}}{(j-k)!} = \text{ coeff. of } x^j \text{ in } F(x)e^{- a_1 x}.
\]
\end{proof}

\end{document}